\documentclass[11pt]{amsart}
\usepackage{amssymb,amsmath,amsthm,latexsym}
\usepackage{amssymb,graphicx}
\newtheorem{theorem}{Theorem}

\newtheorem{conjecture}[theorem]{Conjecture}
\newtheorem{corollary}[theorem]{Corollary}

\newtheorem{lemma}[theorem]{Lemma}

\newtheorem{proposition}[theorem]{Proposition}

\newcommand\Z{\mathbb Z}
\begin{document}

\title{On growth of homology torsion in amenable groups} 

\author{Aditi Kar, Peter Kropholler \and Nikolay Nikolov}

\begin{abstract}
Suppose an amenable group $G$ is acting freely on a simply connected  simplicial complex $\tilde X$ with compact quotient $X$. Fix $n \geq 1$, assume $H_n(\tilde X, \mathbb{Z})=0$ and let $(H_i)$ be a Farber chain in $G$. We prove that the torsion of the integral homology in dimension $n$ of $\tilde{X}/H_i$ grows subexponentially in $[G:H_i]$.  By way of contrast, if $X$ is not compact,  there are solvable groups of derived length 3 for which torsion in homology can grow faster than any given function. 
\end{abstract}

\maketitle
\section{Introduction}

Let $G$ be a group acting freely on a simply connected  simplicial complex $\tilde X$ with compact quotient $\tilde X /G=X$. Let $n \in \mathbb N$. The $n$-skeleton $X^{(n)}$ of $X$ is finite, $\pi_1(X)=G$ and $\tilde X$ is the universal cover of $X$. Let $C_n(X)$ be the free abelian group with basis the set of simplices in dimension $n$ of $X$.  The boundary maps $\delta_n: C_n(X) \rightarrow C_{n-1}(X)$ define the homology groups $H_n(X,\mathbb Z)= \ker \delta_n / \mathrm{im } \ \delta_{n+1}$. Denote the size of the torsion subgroup of $H_n(X, \mathbb Z)$ by $T_n(X)$.

Let $(H_i)$ be a chain of finite index subgroups of $G$ ordered by inclusion. Let $\Omega$ be the profinite space $\underleftarrow{\lim}\  G/H_i$. The counting measure on the finite sets $G/H_i$ induces a probability measure $\mu$ on $\Omega$ and $G$ acts on $\Omega$ by measure preserving homeomorphisms. The chain $(H_i)$ is called \emph{Farber} if this action of $G$ is essentially free, i.e. $Stab_G(x)=1$ for all $x \in \Omega$ except a null set. This is equivalent to saying, for any $g \in G \backslash \{1\}$,  $$\lim_{i \rightarrow \infty} 
\frac{|\{  x \in G/H_i  \ | \  gx=x \} |}{[G:H_i]}=0 .$$

An example of a Farber chain  $(H_i)$ is when each $H_i$ is normal in $G$ and $\bigcap_i H_i= \{1\}$.

\begin{theorem} \label{main} Let $G$ be an amenable group acting freely on a simply connected simplicial complex $\tilde X$ as above with compact quotient $X$. Let $n \geq 1$ and assume that $H_n(\tilde X, \mathbb Z)=0$.  Let $(H_i)$ be a Farber chain
in $G$ and let $Y_i=\tilde X/H_i$ be the corresponding covering space $Y_i$ of $X$.
Then,  

\[ \lim_{i \rightarrow \infty} \frac{\log T_n(Y_i)}{[G:H_i]}=0\]
\end{theorem}

If $G$ is of type $F_{n+1}$ we can take $X$ to be the $n+1$ skeleton of $K(G,1)$ and $\tilde X$ to be the universal cover of $X$, then $H_n(X,\mathbb Z)$ is simply $H_n(G, \mathbb Z)$ and we write $T_n(G)$ for the size of the torsion subgroup of $H_n(G, \mathbb Z)$. Since $\tilde X$ is $n$-connected, Theorem \ref{main} yields the immediate 

\begin{corollary} \label{cor} Let $G$ be an infinite amenable group of type $F_{n+1}$ and let $H_i$ be a Farber chain in $G$.
Then,  

\[ \lim_{i \rightarrow \infty} \frac{\log T_n(H_i)}{[G:H_i]}=0\]
\end{corollary}

In the case when $G$ is elementary amenable or more generally, $G$ contains an infinite normal elementary amenable subgroup, this has been proved by W. L\"uck \cite{Lueck2013}. There is a lot of interest in computing the above limit $\frac{\log T_n(H_i)}{[G:H_i]}$ for arithmetic lattices in semisimple Lie groups, see \cite{berg}, \cite{AGN} and its relationship to analytical torsion, \cite{Lueck2}.  We consider here, the case when $G$ is amenable and provide an elementary proof of the vanishing of torsion growth. In the case when $X$ is an aspherical manifold and $(H_i)$ is an exhausting normal chain Theorem \ref{main} is a special case of \cite[Theorem 1.6]{sauer}, see also \cite{frigerio} for a different approach. It will be interesting Theorem \ref{main} to the case when $G$ has an infinite normal amenable subgroup. Another natural direction for further study is to relax the conditions on $\tilde X$ and relate the growth of torsion in $H_n(Y_i,\mathbb Z)$ to invariants of $\tilde X$ in the spirit of \cite{ThomLi}.
\bigskip

In Corollary \ref{cor}, the hypothesis that the amenable group $G$ is of type $F_{n+1}$, is indeed necessary even for $n=1$. By way of contrast we show that if $G$ is not finitely presented and $M<G$ with $[G:M]$ finite, then the rate of growth of $T_1(M)= | tor H_1(M,\mathbb Z)|$ is not  bounded by any function of $[G:M]$.
\begin{theorem} \label{grow} Let $f: \mathbb N \rightarrow \mathbb N$ be any function. There exists a finitely generated residually finite group $G$, together with a normal chain  of  subgroups of finite index $G>M_1>M_2> \cdots $ such that $\bigcap_i M_i = \{1\}$ and 
\[ T_1(M_i) > f([G:M_i]), \quad \forall i \in \mathbb N. \]

In fact $G$ can be taken to be a solvable group of derived length 3 or an extension of an abelian group by Grigorchuk's $2$-group of subexponential growth.
\end{theorem}

Our construction is very general and gives examples of the form $A \rtimes Q$ with $A$ abelian for any residually finite group $Q$ which has an infinite ascending chain of subgroups  closed in the profinite topology. This raises the natural question: which residually finite groups satisfy the ascending chain condition on closed subgroups?

We propose the following

\begin{conjecture}
Every finitely generated amenable residually finite  group which satisfies the ascending chain condition on closed subgroups is virtually solvable of finite rank.
\end{conjecture}

We offer the following special case as evidence.

\begin{theorem}
\label{sol}
Let $G$ be an elementary amenable group which is finitely generated and residually finite. Then the following are equivalent:
\begin{enumerate}
\item
$G$ is virtually solvable of finite rank.
\item
$G$ has the maximal condition on subgroups that are closed in the profinite topology.
\end{enumerate}
\end{theorem}

Related to this is the theorem of Jeanes and Wilson that finitely generated solvable groups in which every subgroup is closed in the profinite topology are polycyclic, see \cite{JW}.

The finitely generated solvable groups that have finite rank are all \emph{minimax} in the sense that they have a series
$1=G_0\triangleleft G_1\triangleleft\dots\triangleleft G_n=G$ of finite length in which the factors are cyclic or quasicyclic. A \emph{quasicyclic} group is a group of $p$-power roots of unity in $\mathbb C$, and there is one such group up to isomorphism for each prime $p$, denoted by $C_{p^\infty}$. 

Theorem \ref{main} is proved in section \ref{prf} modulo a technical result on F{\o}lner sets which is proved in Section \ref{amenable}. Theorem \ref{grow} is proved in Section \ref{solv} and Theorem \ref{sol} is proved in Section \ref{proofsol}.

\section{Proof of Theorem \ref{main}} \label{prf}
We  find presentations $\langle X_i | R_i \rangle $ of the subgroups $H_i$ of $G$ such that $|X_i|/|G:H_i| \rightarrow 0$ as $i \rightarrow \infty$ while at same time the lengths of all relations of all $R_i$ stay bounded.
In case $n=1$ the bound on the torsion of $H_i^{ab}$ follows from the following fundamental lemma. For a vector $v \in \mathbb Z^m$, by $||v||$ we denote the $l_1$-norm of $v$, i.e. the sum of the absolute values of the coordinates of $v$.

\begin{lemma} \label{tor}Let $m,k,l \in \mathbb N$ and let $v_1, \ldots, v_l \in \mathbb Z^m$ with $||v_i|| \le k$ for all $i=1, \ldots, l$. Let $A=\mathbb Z^m/ \sum_{i=1}^l \mathbb Z v_i$. Let $tor(A)$ denote the size of the torsion subgroup of $A$. Then $tor(A) \leq k^{m}$.
\end{lemma}

\proof Let $L=(v_1,\ldots, v_l)$ be the $m \times l$ matrix having $v_i$ as column vectors. Let $r$ be the rank of $L$. Now  $tor(A)$ is the greatest common divisor $\Delta_r(L)$ of all determinants of $r \times r$ minors of $L$, see \cite[Theorem 3.9]{Jacobson}. This can easily be proved directly: The matrix $L$ can be transformed to a diagonal matrix $L'$ using row and column operations. Then $tor(A)$ is the product of the non-zero $r$ entries on the diagonal of $L'$ which is just $\Delta_r(L')$. On the other hand the row and column operations do not change $\Delta_r$ and so $\Delta_r(L)=\Delta_r(L')=t(A)$. 

To prove Lemma \ref{tor} it is thus sufficient to bound the size of just one determinant of a $r \times r$ minor. So we need to prove the inequality under the assumption that $l=m=r$ and that $A$ is torsion i.e. that the determinant of the matrix $L=(v_1, \ldots, v_m)$ is non-zero. We will show that
$| \det L| \leq k^{m}.$

Let $v_i= (\nu_{j,i})$ with $\nu_{j,i} \in \mathbb Z$ be a column vector in $\mathbb Z^m$.
Now \[|\det L| \leq \sum_{\pi \in S_n} \prod_{i=1}^m |\nu_{i, \pi(i)}|  \leq \prod_{i=1}^m ||v_i|| \leq k^m.\]
\endproof

\subsection*{\emph{Folner} choice of coset representatives} Let $B\subset V$ be a subset of the vertices $V$ of a graph $\Gamma$. By $\partial(B)$, we denote the set of edges joining vertices from $B$ to vertices in its complement $V \backslash B$. Also, $\Gamma_B$ denotes the induced graph on $B$ from $\Gamma$ (i.e. the vertices in $B$ and those edges of $\Gamma$ whose ends are in $B$.) 

The Cayley graph $\Gamma(G,S)$ of a group $G$ with respect to a generating set $S$ is defined to have vertices $G$ and edges $(x,sx)$ for $x \in G$ and $s \in S \cup S^{-1}$. For $\epsilon >0$ a finite subset $B \subset G$ is said to be $(\epsilon,S)$- F{\o}lner if \[ |\partial (F_i)| |F_i|^{-1} < \epsilon. \] The following Theorem allows for a F{\o}lner choice of coset representatives for a Farber chain of subgroups in $G$.

\begin{theorem}
 \label{reps}
 Let $G$ be an infinite amenable group generated by a finite set $S$ and let $\Gamma= \Gamma(G,S)$ be the Cayley graph of $G$ with respect to $S$. Let $(H_i)$ be a Farber chain in $G$. There exists a sequence of subsets $F_i \subset G$ such that $F_i$ is a set of left coset representatives for $H_i$ in $G$ and $|\partial (F_i)| |F_i|^{-1} \rightarrow 0$ with $i \rightarrow \infty$.
\end{theorem}
 This Theorem was proved in \cite{Weiss} (see also \cite{AbertNikolov} for a different argument) in the case when $(H_i)$ is a normal exhausting chain and it implies the existence of a F{\o}lner sequence of monotiles in $G$. In section \ref{amenable} we give an elementary proof of Theorem \ref{reps} based on \cite{Weiss} and the Ornstein-Weiss quasi-tiling Lemma from \cite{OW}. 

The following Lemma allows us to assume that all sets $F_i$ in the above theorem are such that the induced graphs $\Gamma_{F_i}$ are connected.

\begin{lemma}
\label{connect}
Let $F$ be a set of left coset representatives for a subgroup $H$ of finite index in a group $G$ and let $\Gamma$ be a Cayley graph of $G$. There is another set $L$ of coset representatives for $H$ in $G$ such that $\Gamma_L$ is connected and $|\partial(L) |\leq |\partial (F)|$.
\end{lemma}

\proof The group $G$ acts on the right on its Cayley graph $\Gamma$ and we have $G=FH$. If $\Gamma_F$ is connected then take $L=F$. Otherwise let $C_1, C_2, \ldots, C_s $ be the connected components of $\Gamma_F$. Note that for $i\neq j$, the boundary of $C_i$ does not share a common edge with the boundary of $C_j$ and consequently, $|\partial(F)|= \sum_j |\partial(C_j)|$. Since $\Gamma= \cup_j C_jH$ is connected there is some edge joining a vertex $l \in C_ih$ with a vertex $w \in C_jg$ for some $g,h \in H$ and $i \neq j$.
Now consider $F'= (\cup_{j \not = i} C_j) \cup C_i'$ where $C_i'= C_ihg^{-1}$. This is another set of coset representatives for $H$ in $G$. Since $C_i'$ and $C_j$ share the edge $(lg^{-1},wg^{-1})$ on their boundary the induced graph on $C_i' \cup C_j$ is connected and the boundary of $C_i' \cup C_j$ has fewer edges than $|\partial(C_i)|+ |\partial(C_j)|$. Therefore, $|\partial(F')| < |\partial(F)|$ and we can replace $F$ with $F'$. Continuing in this fashion we reduce the number of connected components of $\Gamma_F$ until it becomes connected.
\endproof

\subsection*{Proof of Theorem \ref{main}: $n=1$} We first prove Theorem \ref{main} in the case of $n=1$. The argument, though similar to the general case, provides a more effective combinatorial alternative because it uses graphs rather than the chain complex. The reader who wants to see a homological proof of Theorem \ref{main} may proceed directly to the general case. 

Let $\langle S,R\rangle$ be a presentation of $G$ and let $k$ be the length of the longest word in $R$.

Note that for $n=1$ the group $H_1(Y_i, \mathbb Z)= H_1(H_i,\mathbb Z)=H_i^{ab}$ is the abelianization of $H_i$ and thus depends only on $H_i$. So we may use any $CW$-complex $X$ with fundamental group $G$ and its universal cover $\tilde X$ in order to describe $H_i^{ab}$ and its torsion subgroup. 

 Let $X$ be the presentation complex of $G$ and let $\tilde X$ be the universal cover of $X$. We choose a base vertex $v \in \tilde X$ so that the vertices of $\tilde X$ are in bijection with the orbit $vG$ of $v$ under the action of $G$. 
Let $U_i=vF_i$ be the subset of the vertices of $\tilde X$ which corresponds to $F_i$ and let $\Gamma_{U_i}$ be the subgraph induced on $U_i$ from $\tilde X$. By Lemma \ref{connect} we may assume that $\Gamma_{U_i}$ is connected. Note that the boundary of $\Gamma_{U_i}$ has size exactly $\partial(F_i)$. Let us take a spanning tree $T_i$ of $\Gamma_{U_i}$. Then $T_i$ has $|F_i|$ vertices  and the covering map $p: \tilde X \rightarrow \tilde X/H_i=Y_i$ defines a bijection between the vertices of $T_i$ and $Y_i$. If $\eta$ is the number of edges in the graph $\Delta_i:=p(T_i)$, then $|F_i|-1 \leq \eta$; but, as $\eta$ is no larger than the number of edges in $T_i$, $\eta \leq |F_i|-1$. This means that $\eta=|F_i|-1$ and therefore $\Delta_i$ is a spanning tree for $Y_i$. Let $y= p(v)$ be the base vertex of $Y_i$. The group $H_i$ is isomorphic to $\pi_1(Y_i)$ and therefore has presentation $\langle E  \ | D \rangle$ with the following description. For an edge $e=(a,b) \in Y_i^1$ let $l(e)$ be the closed path which travels from $y$ to $a$ on $\Delta_i$, then on the edge $(a,b)$ and returns from $b$ to $y$ on $\Delta_i$. The generating set $E$ is the set of paths $l(e)$ where $e$ is an edge of $Y_i$ outside $\Delta_i$. The relations in $D$ are given by the closed paths on the disks of $Y_i^2$. We note that the relations of $D$ have the same length as the boundary of the disks of $X^2$, i.e. the relations in $R$. 

Let  \[ E'= \{ l(p(\tilde e)) | \ \tilde e \textrm{ is an edge of } \Gamma_{U_i} \} \] and $E''= E \backslash E'$. Every edge $e$ of $Y_i^1$ is equal to $p(\tilde e)$ for  some edge $\tilde e$ of $\tilde X$ with at least one vertex in $U_i$. Moreover by definition $l(e) \in E'$ if and only if both ends of $\tilde e$ are in $U_i$. 
We see that $|E''| \leq | \partial F_i |$  which is the number of edges of $\tilde X$ with one end in $U_i$ and the other outside $U_i$. When $\tilde e$ is an edge of $\Gamma_{U_i}$ let $\tilde l( \tilde e)$ be the closed path in $\tilde X$ which travels from $v$ on $T_i$ to one end of $\tilde e$ and then returns  from  the other end of $\tilde e$ back to $v$ on the $T_i$.  Note that $l(p(\tilde e))=p(\tilde l ( \tilde e))$. Since $\tilde X$ is simply connected, the path $\tilde l ( \tilde e)$ is null-homotopic and hence so is $p(\tilde l( \tilde e))$. Therefore the generators $E'$ represent the trivial element in $H$.

We see that $H_i= \langle E \ | \ D \rangle = \langle E \ | \ D \cup E' \rangle  \simeq\langle E'' | \  \bar D \rangle$ where $\bar D$ is the image of the set $D$ under the homomorphism from the free group on $E$ to the free group on $E''$ which sends all elements of $E'$ to 1.

Observe that all the relations of $\bar D$ still have length at most $k$. By abelianizing the presentation $\langle E'' | \  \bar D \rangle$ of $H_i$ we are in a position to apply Lemma \ref{tor} and deduce that  
\[ |tor H_i^{ab}| \leq k^{|E''|}  \leq k^{|\partial(F_i)|} \]

This holds for each $i$. Moreover  $H_i^{ab}= H_1(Y_i,\mathbb Z)$,  and $|\partial(F_i)|/|G:H_i|=|\partial(F_i)|/|F_i|$ which tends to zero as $i \rightarrow \infty$. This proves the case $n=1$ of Theorem \ref{main}.

\subsection*{Proof of Theorem \ref{main}: General case} 
For the general case we argue directly with the homology groups $H_n(H_i,\mathbb Z)$.  Choose $\mathcal D$ to be a collection of simplices of $\tilde X$ which is a fundamental domain for the action of $G$ on $\tilde X$.  Let $H_i,F_i, Y_i$ and $U_i$ be as above. Let 
\[ \mathcal F_i =\bigcup_{g\in F_i} \mathcal Dg .\]
Then $\mathcal F_i$ is a fundamental domain for the action of $H_i$ on $\tilde X$. 
For a union of simplices $A \subset \tilde X$ we denote by $\bar A$ the topological closure of $A$ in $\tilde X$, this is also a union of simplices. Define
\[ \mathcal J_i= \bigcup \left \{ \mathcal Dg \ | \ \overline{Dg} \subset \mathcal F_i\right \} \]

and let  $J_i= \{ g \in G \ | \ \mathcal D g \subset \mathcal J_i\}$.

 Since $\tilde X$ is quasi-isometric to the Cayley graph $\Gamma$ of $G$ we see that
\begin{equation}  \label{eq1} \frac{| F_i \backslash J_i|}{|F_i|} \rightarrow 0 \quad \textrm{as }i \rightarrow \infty .\end{equation}
As a consequence if $m_i$ is the number of simplices of $\mathcal B_i:=\mathcal F_i \backslash \mathcal J_i$ then (\ref{eq1}) gives that $m_i/|F_i| \rightarrow 0$ with $i \rightarrow \infty$.

Recall that $C_n(Y_i)$ denotes the free abelian group with basis $Y_i^{(n)}$. 
For each simplex $ c \in Y_i^{(n)}$, let $\tilde c$ be the unique simplex of $\mathcal F_i$ such that $p(\tilde c)=c$. Note that $Y_i$ is the disjoint union of $p(\mathcal J_i)$ and $p(\mathcal B_i)$

Let $V_i, W_i$ be the subgroups of $C_n(Y_i)$ generated by the $n$-simplices of $p(\mathcal J_i)$ respectively of $p(\mathcal B_i)$. We have $C_n(Y_i)=V_i \oplus W_i$ and the relative homology group $H_n(Y_i,p(\mathcal J_i), \mathbb Z)$ is defined to be
\[H_n(Y_i,p(\mathcal J_i), \mathbb Z)=\frac{\ker \delta_n+ V_i}{\mathrm{im} \ \delta_{n+1}+V_i}\] There is a homomorphism $f: H_n(Y_i,\mathbb Z) \rightarrow H_n(Y_i,p(\mathcal J_i), \mathbb Z)$ given by \[ f(x + \mathrm{im}\ \delta_{n+1})= x+ \mathrm{im} \ \delta_{n+1}+V_i \] for each $x \in \ker \delta_n$.

The covering map $p: \tilde X \rightarrow Y_i$ defines a homomorphism $p_n : H_n(\tilde X, \mathbb Z) \rightarrow H_n(Y_i, \mathbb Z)$ which sends a cycle $\sum_j  \lambda_j \tilde c_j$ with $\lambda_j \in \mathbb N, \tilde c_j \in C_n(\tilde X)$  to the cycle $\sum_j \lambda_j p(\tilde c_j) \in C_n(X)$. Let $H_n(\tilde X)_{\mathcal J_i}$ be the subgroup of $H_n(\tilde X, \mathbb Z)$ generated by the images of cycles with support in $\mathcal J_i$.

\begin{lemma}\label{exact}  The following sequence of abelian groups is exact at $H_n(Y_i, \mathbb Z)$.
\[ H_n(\tilde X)_{\mathcal J_i}\stackrel{p_n}{\longrightarrow} H_n(Y_i, \mathbb Z) \stackrel{f}{\longrightarrow} H_n(Y_i,p(\mathcal J_i),\mathbb Z).\]
\end{lemma}

\begin{proof} Let us verify the exactness at $H_n(Y_i, \mathbb Z)$. Suppose $x + \mathrm{im} \ \delta_{n+1} \in \ker f$. Then $x=y+v$ where $v \in V_i$ and $y \in \mathrm{im} \ \delta_{n+1}$. Hence $v=x-y \in \ker \delta_n \cap V_i$ is a cycle of $C_n(Y)$ supported on the $n$-simplices of $p(\mathcal J_i)$.  It is sufficient to prove that $v =p_n(\tilde v)$ for some cycle $\tilde  v$ in $C_n(\mathcal J_i)$. Suppose that $v = \sum_{j} \lambda_j c_j$ where $\lambda_j \in \mathbb Z$ and $c_j \in p(\mathcal J_i) \subset Y_i$. Then $\tilde v := \sum_j \lambda_j \tilde c_j \in C_n(\mathcal J_i) < C_n(\tilde X)$. From the definition of $\mathcal J_i$,  the boundary of each $\tilde c_j$ is inside $\mathcal F_i$  and so $\delta_n (\tilde v) \in C_{n-1}(\mathcal F_i)$. Since $v$ is a cycle of $C_n(Y_i)$ we have
\[ p(\delta_n(\tilde v))= \delta_n(p(\tilde v))= \delta_n(v)=0 .\] On the other hand, $\mathcal F_i$ is a fundamental domain for $H_i$ and from $\delta_n(\tilde v)  \in C_{n-1}(\mathcal F_i)$ and $p(\delta_n(\tilde v))=0$ we obtain
 $\delta_n(\tilde v)=0$ i.e. $\tilde v$ is a cycle in $C_n(\mathcal J_i)$, hence $v=p_n(\tilde v) \in \mathrm{im}\ p_n$ and the sequence is exact at $H_n(Y_i, \mathbb Z)$.

\end{proof}
\medskip

By assumption, $H_n(\tilde X, \mathbb Z)=0$, so Lemma \ref{exact} implies that 
$H_n(Y_i, \mathbb Z)$, which is isomorphic to $H_n(Y_i,p(\mathcal J_i),\mathbb Z)$, is a subgroup of $\frac{C_n(Y_i)}{ V_i+ \mathrm{im} \delta_{n+1}}$. In turn by considering the projection $h: C_n(Y_i)= W_i \oplus V_i \rightarrow W_i$ we see
\[ \frac{C_n(Y_i)}{ V_i+ \mathrm{im} \delta_{n+1}} \simeq W_i/W_{i,0} \]
where $W_{i,0}=h(\mathrm{im} \delta_{n+1})$. 
The group $W_i$ is a free abelian group generated  by the $n$-simplices  in $p(\mathcal B_i)$ and in particular has a basis of size at most $m_i$.

Note that $\mathrm{im} \ \delta_{n+1}$ is generated by  boundaries of simplices in $C_{n+1}(Y)$ and each such boundary is  an integral combination of $n+2$ simplices of dimension $n$. Since $h$ is a projection onto a direct summand spanned by a subset of the basis $Y_i^{(n)}$ we conclude that $W_{i,0}=h(\mathrm{im} \ \delta_{n+1})$ is spanned by vectors of $l_1$-norm at most $n+2$.

Lemma \ref{tor} applies and gives that $|tor W_i/W_{i,0}| \leq (n+2)^{|m_i|}$. Since $H_n(Y_i, \mathbb Z)$ is isomorphic to a subgroup of  $W_i/W_{i,0}$  we obtain \[ tor (H_n(Y_i, \mathbb Z)) \leq tor (W_i/W_{i,0}) \] and so $\log |tor (H_n(Y_i, \mathbb Z)| \leq |m_i| \log (n+2)$.
Together with the earlier observation that $|m_i|/[G:H_i] \rightarrow 0$ with $i \rightarrow \infty$ this completes the proof of Theorem \ref{main}. 

\section{Quasi-tilings of amenable groups} \label{amenable}

In this section we prove Theorem \ref{reps}.

A \emph{monotile} for a group $G$ is a finite set $T$ such that there exists a subset  $C \subset G$ such that
$\cup_{c \in C}{Tc}$ is a partition of $G$. The relevance of monotiles to our proof is provided by the following 

\begin{theorem}[\cite{OW}] \label{qt}
Let $G$ be an amenable group with a p.m.p essentially free action on a probability space $(\Omega, \mu)$. Let $T \subset G$ be a monotile and let $\epsilon >0$. Then there exist a measurable subset $A \subset \Omega$ such that

1. $t_1 A \cap t_2 A= \emptyset$ for all $t_1 \not = t_2 \in T$, and

2. $\mu ( \cup_{t \in T} tA) > 1- \epsilon$.
\end{theorem}

Of course this is only useful if one has a good source of monotiles. It is not known if every amenable group has a F{\o}lner sequence of monotiles. When $G$ is residually finite \cite{Weiss} B. Weiss constructed such a sequence of monotiles $T_i$ with  $|\partial (T_i)| |T_i|^{-1} \rightarrow 0$. In fact the sets $T_i$ are chosen to be coset representatives for a normal exhausting chain in $G$. To obtain F{\o}lner representatives for Farber chains we need to use Theorem \ref{qt}. 

So let $H_i$ be a Farber chain in $G$ and let $\Omega = \underleftarrow{\lim} \ G/H_i$ with probability measure $\mu$ induced from the counting measure on each $G/H_i$. By $U_i$ we denote the closure of $H_i$ in $\Omega$, this is an open set of measure $[G:H_i]^{-1}$. The base of the topology in $\Omega$ is given by all $\mathcal  B=\{ gU_i \ | \ g \in G, i \in \mathbb N\}$.

\begin{proof}[Proof of Theorem \ref{reps}.]
For $\delta >0$ let $T$ be a $(\delta,S)$-invariant monotile of size $m$ provided by \cite{Weiss}.

We are going to prove that there is an integer $N$ such that for all $n>N$ there is a $(2\delta,S)$-invariant set $F \subset G$ which maps bijectively onto $G/N_n$, i.e. forms a set of left coset representatives for $H_n$ in $G$.

Let $\epsilon >0$ and let $A \subset \Omega$ be the subset provided in Theorem \ref{qt} for this $T$, $\Omega$ and $\epsilon$. We can find a compact set $K \subset G $ and an open $O \subset G$ such that $K \subset A \subset O$ and $\mu(O \backslash A)< \epsilon$ and $\mu(A \backslash K)< \epsilon$. Let $\mathcal B_O$ be the elements of $\mathcal B$ contained in $O$. Since $K \subset O$ we see that $\mathcal B_O$ is an open cover of the compact set $K$. Hence there exists a finite subcover of $K$ which implies that there is an integer $N$ (depending on $O$ and $K$) such that for all $n>N$ there exist $x_1, \ldots x_k \in G$ (depending on $n$), such that \[ K \subset \cup_{i=1}^k x_iU_n \subset O. \] We may assume that $x_iU_n \not = x_jU_n$  if $i \not = j$.

 Let 
$A'=\cup_{i=1}^k x_iU_n$ and let $a = \mu(A')$. Note that since 
\[ \mu(\cup_{t \in T} tK)> 1-\epsilon -m\epsilon  \] and $tA'$ still covers $tK$ for each $t \in T$ we have
\[ a=\mu( \cup_{t,i} tx_i U_n) > 1-(m+1) \epsilon. \]

Also since $\mu(O) \leq 1/m + \epsilon$ we get $k[G:H_n]^{-1}=\mu ( A') \leq 1/m + \epsilon$ and therefore 
$k \leq [G:H_n](1/m +\epsilon)$.
  
Let $F'= \cup_{i=1}^k Tx_i$. Note that \[ |\partial F'| \leq k |\partial T| \leq \delta |T| [G:H_n](1/m +\epsilon)=\delta [G:H_n](1+m \epsilon). \] On the other hand the image $\bar F'$ of $F'$ in $G/H_n$ has size 
\[| \bar F'|= [G:H_n]a >[G:H_n](1-(m+1) \epsilon).\]


Let $y:=|T| k - | \bar F'|$, we have 
\[ y< [G:H_n] (1+m\epsilon)- [G:H_n](1-(m+1)\epsilon)=[G:H_n] (2m+1) \epsilon. \] 

Let \[ z:= [G:H_n] -|\bar F'|< [G:H_n] (m+1) \epsilon. \] We can remove $y$ elements from $F'$ and add $z$ elements to it so that the resulting set $F$ maps bijectively onto $G/H_n$. This means that $F$ is a set of left coset representatives for $H_n$ in $G$. The boundary of $F$ can increase by at most $2|S|(y+z)< 2|S|(3m+2)[G:H_n]\epsilon$.

Now $|\partial F|$ is at most

\[ |\partial F'|+ 2|S|(3m+2)[G:H_n] \epsilon < \delta [G:H_n](1+\epsilon (m+2|S|(3m+2) ).\] Let $\epsilon$  be chosen such that $\epsilon < (m+2|S|(3m+2))^{-1}$. We obtain $|\partial F|< 2 \delta [G:H_n]=2\delta |F|$. We constructed such a set $F$ for each $n>N$. The Theorem follows. \end{proof}
\medskip

\section{Proof of Theorem \ref{grow}} \label{solv}

We shall adopt the convention that for a commutative ring $S$ and group $G$ or $H$ or $K$, the augmentation ideal of the group algebra is denoted by the corresponding German fraktur letter $\mathfrak g$ or $\mathfrak h$ or $\mathfrak k$.

We shall construct our example $G$ as a split extension (semidirect product) of an abelian normal subgroup $W$ by a group $Q$. 
We write $W$ additively, regarding it as a $\Z Q$-module.
Thus $G=Q\ltimes W$ consists of ordered pairs $Q\times W$ with product $(q,w)(q',w')=(qq',wq+w')$.

The construction that follows is possible when $Q$ fails to have the ascending chain condition on subgroups that are closed in the profinite topology, and since this happens for many metabelian groups we are able to construct a $G$ that is soluble of derived length $3$.

\begin{lemma}\label{failsascendingcondition}
Let $Q$ be a residually finite group which possesses an infinite locally finite subgroup. Then $Q$ fails to have the ascending chain condition on subgroups that are profinitely closed. In particular, the lamplighter groups $C_p\wr\Z$ ($p$ prime) provide examples of this behaviour.
\end{lemma}

\begin{proof}
In a residually finite group, the finite subgroups are closed, and in a group with an infinite locally finite normal subgroup there are strictly ascending chains of finite subgroups.
\end{proof}

There are many more residually finite groups for which this ascending chain condition fails including all branch groups  (see Propositions \ref{branch} below), and all finitely generated residually finite elementary amenable groups (Theorem \ref{sol}).

If $R$ is a subgroup of $Q$ and $T$ is a submodule of $W$ then $R\ltimes T$ is a subgroup of $G$. In this notation we have the following, the proof of which we leave to the reader.

\begin{lemma}\label{abelianization}
Let $R$ be a normal subgroup of $Q$. Then the derived subgroup of $R\ltimes W$ is $[R,R]\ltimes W\mathfrak r$, and the derived factor group $H_1(R\ltimes W,\Z)$ is isomorphic to the direct product $R/[R,R]\times W/W\mathfrak r$.
\end{lemma}

Let $Q$ be a finitely generated residually finite group which possesses a strictly ascending sequence $1=P_0<P_1<P_2<\dots$ of subgroups each of which is closed in the profinite topology. Set $P=\bigcup_iP_i$. 
It is possibly worth noting that we do not require $P$ to be closed.

The construction of the module $W$ depends on a choice of sequence $(n_i)$ of natural numbers $\ge2$. 
We shall describe a general construction first which produces a module for any sequence and only towards the end of the argument will we choose a particular sequence to obtain the required fast growth rate. 
Let $(n_i)$ be a sequence 
 and let $V_i$ be the $\Z P$-module $\Z/n_i\Z[P_i\backslash P]$ and let $V_i'$ be the kernel of the augmentation map
$V_i\to\Z/n_i\Z$. Then $V_i$ is the permutation module on the cosets of $P_i$ in $P$ having exponent $n_i$, and we write $e_i$ for the generator $P_i$. Let $\overline V$ be the product $\prod_iV_i$ and let  $U$ be the $\Z P$-submodule of $\overline V$ generated by
$e:=(e_0,e_1,e_2,\dots)$. Let $V$ be the direct sum $\bigoplus_iV_i$. 
The module $W$ is now defined to be the induced module $U\otimes_{\Z P}\Z Q$ and the role of $V$ is explained in Lemma \ref{modulestructure} below.

We have the following properties of the $\Z P$-modules $V$, $\overline V$ and $U$. 

\begin{lemma}\label{modulestructure} 
Assume that the sequence $(n_i)$ is unbounded.
\begin{enumerate}
\item
$U\cap V$ is the submodule of $U$ consisting of the $\Z$-torsion and $U/U\cap V$ is a trivial cyclic module isomorphic to $\Z$.
\item
 $V\cap U$ is the direct sum $\bigoplus_iV_i'$.
\end{enumerate}
\end{lemma}
\begin{proof}
Let $p$ be an arbitrary element of $P$. Then there exists $j$ such  that $p\in P_j$. Therefore $e_ip=e_i$ for all $i\ge j$. Hence $e(p-1)$ has only finitely many non-zero coordinates and so lies in $V$. Thus $U\mathfrak p\subseteq V$ and $U/U\mathfrak p$ is a cyclic trivial module which is isomorphic to $\Z$ because the $n_i$ are unbounded. 
It is also clear now that $U\cap V=U\mathfrak p$ is precisely the torsion submodule of $U$.
This proves (1). Since $U\mathfrak p=U\cap V$, we see that the projection of $U\cap V$ onto any $V_i$ has image contained in $V_i'$.

We now show by induction on $j$ that $e\mathfrak p_j\Z P=V_0'\oplus\dots\oplus V_{j-1}'$ and then (2) follows. First, the $\Z P_j$ module $e\mathfrak p_j$ is contained in $V_0\oplus\dots\oplus V_{j-1}$ because $P_j$ fixes $e_k$ for all $k\ge j$. Inductively we may assume that $e\mathfrak p_{j-1}\Z P=V_0'\oplus\dots\oplus V_{j-2}'$ from which it follows that
$e\mathfrak p_{j}\Z P$ contains $V_0'\oplus\dots\oplus V_{j-2}'$ as a submodule. Since $e\mathfrak p_j$ also maps surjectively to $\Z/n_j\Z$ on projection to the $j$-th factor the result follows.
\end{proof}

Let $Y_i= V_i \otimes_{\mathbb{Z}P} \mathbb{Z}Q$ and $Y_i ' = V_i' \otimes_{\mathbb{Z}P} \mathbb{Z}Q$; set $Y = \oplus_i Y_i$ and $\bar Y= \prod_i Y_i$. Then $W $ is a $\Z [Q]$ submodule of $\bar Y$. Note that by Lemma \ref{modulestructure}, $W \cap Y= \oplus_i Y_i'$. The following Lemma identifies torsion in $\mathbb{Z}[Q]$-submodules of $W$ that contain $Y_i'$.

\begin{lemma}\label{derivedsubgp} Fix $i\ge0$. 
If $R$ is a normal subgroup of finite index in $Q$ such that $P_iR<PR$ and if $X$ is a $\Z Q$-submodule of $W$ which contains $Y_i'$, then $tor(X/X\mathfrak r) \geq n_i$.
\end{lemma}
\begin{proof} 

The torsion subgroup of $\frac{X}{X\mathfrak{r}}$ contains $\frac{Y_i'}{Y_i' \cap X \mathfrak{r}}$. As $Y_i' \cap X\mathfrak{r}$ is contained in $Y_i\mathfrak{r}$, it is sufficient to estimate $\frac{Y_i'}{Y_i' \cap Y_i\mathfrak{r}} \cong \frac{Y_i'+Y_i\mathfrak{r}}{Y_i\mathfrak{r}}$. The quotient $\frac{Y_i}{Y_i \mathfrak r}$ is isomorphic to $\Z/n_i \Z[P_i R\backslash Q]$, which is a free $\Z/n_i\Z$-module of rank $|Q:P_iR|$. 

\noindent We now consider $\frac{Y_i}{Y_i'+Y_i \mathfrak{r}}$:  

$$\frac{Y_i}{Y_i'+Y_i \mathfrak{r}} \cong \frac{Y_i/Y_i'}{(Y_i/Y_i')\mathfrak{r}} \cong \Z/n_i \Z[PR\backslash Q]. $$

Therefore, $\frac{Y_i}{Y_i'+Y_i \mathfrak{r}}$ is a free $\Z/n_i\Z$-module of rank $|Q:PR|$. As $P_iR < PR$, we obtain 
$$\left |\frac{Y_i'}{Y_i' \cap X \mathfrak{r}}\right | \geq \left |\frac{Y_i'}{Y_i' \cap Y_i\mathfrak{r}}\right | = n_i^{|Q:P_iR|-|Q:PR|}>n_i.$$ 

\end{proof}

\begin{lemma}\label{rfmodule} Fix $j \in \mathbb N$ and let $(Q_i)$ be a normal chain in $Q$ such that $\bigcap_{i=1}^\infty Q_iP_j=P_j$. Then $\bigcap_{i=1}^\infty Y_j \mathfrak{q}_i=\{0\}$.
In particular $Y_j$ is residually finite as a $\Z Q$-module.
\end{lemma}

\begin{proof} The module $Y_j \cong \Z / n_j \Z [P_j \backslash Q]$ is the permutation module on the cosets $P_j \backslash Q$.  If $x = \sum_{s=1}^l \lambda_s P_jg_s$  with $\lambda_s \in \Z/n_j \Z$ and $g_s \in Q$ is a non-zero element of $Y_j$ where  the cosets $(P_jg_s)_{s=1}^l$ are pairwise distinct we can find $i \in \mathbb N$ such that the cosets $\{P_jQ_ig_s\}_{s=1}^l$ are pairwise distinct. This implies that $x \not \in Y_j\mathfrak{q}_i$.

\end{proof}

\begin{lemma}\label{phk}
Suppose that $f\colon\mathbb N\to\mathbb N$ is a function. Then there exist a choice of strictly ascending sequence $n_0<n_1<n_2<\dots$ and a chain $W_0> W_1>W_2>\dots$ 
of submodules of finite index in $W$ such that $\bigcap_iW_i=0$ and $tor W_i/W_i \mathfrak{q_i}> f(|Q/Q_i|.|W/W_i|)$ for all $i \geq 1$.
\end{lemma}

\begin{proof}
Since $Q$ is residually finite and each $P_j$ is profinitely closed, there exists a chain $Q=Q_0>Q_1>Q_2>\dots$ of normal subgroups of finite index in $Q$ such that  $\bigcap_{i=1}^\infty P_jQ_i=P_j$, $P_jQ_j<PQ_j$ for all $j$, and
$\bigcap_iQ_i=\{1\}$. 

We construct the sequences $n_i$ by induction on $i$. The idea at the inductive step is choose $n_i$ using that $|W:W_i|$ is finite and is a function of only $n_0,\dots,n_{i-1}$. To this end we consider the map 
$$\phi_i: W \rightarrow  Y_0\oplus\dots\oplus Y_{i-1} \rightarrow \frac{Y_0} {Y_0\mathfrak{q_{i-1}}} \oplus \cdots \oplus \frac{Y_{i-1}}{Y_{i-1}\mathfrak{q_{i-1}}}$$ where the first homomorphism is the projection of $W$ onto the first $i-1$ factors of $\bar Y$. The image of $\phi_i$ is finite and we define $W_i= \ker \pi_i$. Lemma \ref{rfmodule} gives $\bigcap_{i=1}^\infty W_i =\{0\}$. Note that the index of $W_i$ divides 
$\prod_{j=0}^{i-1}|\frac{V_j}{V_j \mathfrak{q_i}}|=\prod_{j=0}^{i-1}n_j^{|Q:P_jQ_{i-1}|}$ and in particular does not depend on $n_i$. 

 Note that $W_i$ then contains the whole of $\oplus _{j\ge i}Y_j$. Thus by Lemma \ref{derivedsubgp} we have $tor W_i/W_i \mathfrak{q_i}>n_i$. We now choose $n_i>n_{i-1}$ so that $n_i>
f(|Q/Q_i|.|W/W_i|)$.
\end{proof}

\begin{proof}[Proof of Theorem 3] We define $G$ to be the semidirect product of $W$ by $Q$ where $Q$ is chosen to satisfy the conclusions of Lemma \ref{failsascendingcondition} and $W$ is constructed to satisfy the conclusions of Lemma \ref{phk}. Then the finite generation of $G$ is guaranteed because $Q$ is finitely generated and $W$ is cyclic as a $\Z Q$-module.
Theorem 3 now follows by taking $M_i=Q_i\ltimes W_i$ and using Lemmas \ref{abelianization} and \ref{derivedsubgp} together with the fact that both chains $(Q_i)$ and $(W_i)$ intersect trivially. If $Q$ is a lamplighter group, for example, then this construction produces a solvable group of derived length $3$ as promised. If $Q$ is taken to be Grigorchuk's $2$-group of intermediate growth then we can construct an amenable example that is
abelian-by-(subexponential growth). Both Grigorchuk's group and the lamplighter group have infinite abelian torsion subgroups so Lemma 
\ref{failsascendingcondition} applies.
\end{proof}

 In fact $Q$ above can be any branch group: 

\begin{proposition}\label{branch}
Let $G$ be a weakly branch group acting on a rooted tree $T$. Then $G$ does not satisfy the ascending chain condition on closed subgroups.
\end{proposition}
\begin{proof} We refer the reader to the survey \cite{W} for the definitions and background theory of (weakly) branch groups.
For a vertex $v \in T$ denote by $D(v)$ the vertices of $T$ which are descendants of $v$.
Let us choose a sequence of vertices $(v_j)_{j=1}^\infty$ such that their descendants $D(v_j)$ are pairwise disjoint. Let $g_j \in G$ be a nontrivial element of the rigid stabilizer $rist_G(v_j)$ and let $k_j \in D(v_j)$ be a vertex moved by $g_j$. Let $P_j$ be the pointwise stabilizer of the vertices $\{k_i \ | \ i \geq j \}$ in $G$. Then each $P_j$ is closed in the profinite topology of $G$ and since $ g_j \in  P_{j+1} \backslash P_j$ we have $P_1 <P_2 < \cdots $ is a strictly ascending chain of subgroups.  
\end{proof}
\section{Proof of Theorem \ref{sol}} \label{proofsol}

For the proof of Theorem \ref{sol} we need some preparatory lemmas.
\medskip

We define the \emph{length} $l(G)$ of a solvable minimax group $G$ to be the number of infinite factors in a cyclic/quasicyclic series of $G$. The \emph{Hirsch length} or torsion-free rank, $h(G)$ is the number of infinite cyclic factors in such a series. For the sake of brevity we henceforth use the terminology \emph{closed subgroup} to mean a subgroup that is closed in the profinite topology: that is, any subgroup that is an intersection of subgroups of finite index. We shall use the term \emph{dense} to mean dense in the profinite topology.

The \emph{Pr\"ufer rank} of a group $G$ is the least $d\ge0$ such that every finitely generated subgroup can be generated by at most $d$ elements. 

\begin{lemma}\label{Lemma18}
Let $A$ be a torsion-free abelian minimax group. Then the Pr\"ufer rank of $A$ is equal to the minimum number of (topological) generators of the profinite completion of $A$, and is also equal to the Hirsch length of $A$.
\end{lemma}
\begin{proof}
By the definition of minimax groups $A$ has a finitely generated free abelian subgroup $B$ such that $A/B= \prod_{q \in \pi} C_{q^\infty}$ for a finite, possibly empty set $\pi$ of primes. The Hirsch length of $h=h(A)$ is the rank of $B$. Let $p$ be a prime outside $\pi$.
Now
 $A/A^p\cong B/B^p = (\mathbb Z/ p \mathbb Z)^h$. Therefore the Pr\"ufer rank is at least as great  as $h$. Since $A$ is torsion-free, every finitely generated subgroup is free abelian of rank at most the rank of $B$, that is $h$. Hence the Pr\"ufer rank is equal to $h$.  Let $d$ be the minimum number of generators of the profinite completion of $A$. Then $d\ge h$ because $A$ has a quotient  $(\mathbb Z/ p \mathbb Z)^h$. On the other hand $B$ has rank $h$ and $A/B$ is torsion and divisible and so $B$ is dense in the profinite topology on $A$. Therefore $d\le h$ and the proof is complete.
 \end{proof}
 
 We shall need the following results about minimax groups.

\begin{theorem} \cite[Theorem 10.33]{Robinson}\label{minmax} Let $G$ be a residually finite solvable minimax group. Then 

1. $G$ is nilpotent-by-abelian-by-finite.

2. $G$ has subgroups $\{1\} =G_0 \vartriangleleft G_1 \vartriangleleft \cdots \vartriangleleft G_{k+1}=G$ such that $G/G_k$ is finite and $G_{i+1}/G_i$ is a torsion free abelian minimax group for all $i=0, 1, \ldots, k-1$. 
\end{theorem}

\begin{lemma}\cite[Lemma 2.5]{kw}.\label{Lemma19}
Let $G$ be a residually finite solvable minimax group and let $A$ be an abelian normal subgroup of $G$. Then the subspace topology on $A$ induced by the profinite topology on $G$ coincides with the profinite topology on $A$.
\end{lemma}

\begin{lemma}\label{Lemma20}
Let $G$ be a residually finite group with the maximal condition on closed subgroups. Suppose that $A$ is a closed abelian normal subgroup such that $G/A$ is solvable and minimax. Then there is a unique normal closed subgroup $B$ of $G$ minimal subject to the condition that $A/B$ is torsion-free, and minimax.
\end{lemma}
\begin{proof} Let $\mathcal B$ be the set of all subgroups $B$ of $A$ that are closed and normal in $G$ and such that $A/B$ is torsion free and minimax.
Let $A_0$ be a finitely generated dense subgroup of $A$. Such a subgroup exists because $A$ has the maximal condition on closed subgroups. Let $d$ be the minimum number of generators of $A_0$. Then the closure $\overline A$ of $A$ in the profinite completion $\widehat G$ of $G$ can be (topologically) generated by $d$ elements and hence so is $\overline A/\overline B$ for any $B\in\mathcal B$, where $\overline B$ denotes the closure of $B$ in $\widehat G$.  By Lemma \ref{Lemma19}, $\overline A/\overline B$ is isomorphic to the profinite completion of $A/B$ and hence by Lemma \ref{Lemma18}, the Hirsch length of $A/B$ is at most $d$. Let $B_0 \in \mathcal B$ be a subgroup such that the Hirsch length of $A/B_0$ is largest possible. If $B_1 \in \mathcal B$ with $B_ 1 < B_0$ then  $B_0/B_1$ contains an infinite cyclic subgroup, so $h(B_0/B_1) >0$ and $h(A/B_1) > h(A/B_0)$ contradiction. So $B_0$ is the required minimal subgroup.
\end{proof}

One direction of Theorem 5 is relatively easy to prove and does not require the hypotheses of finite generation or residual finiteness:

\begin{lemma}
Let $G$ be a virtually solvable minimax group. Then $G$ satisfies the maximal condition on closed subgroups.
\end{lemma}
\begin{proof}
Clearly we may assume that $G$ is residually finite and hence by Theorem \ref{minmax}, $G$ has a subgroup of finite index which is poly-(torsion-free abelian minimax). Therefore it suffices to establish the result for abelian minimax groups because the class of groups satisfying the maximal condition on closed subgroups is extension closed. Suppose $(H_j)$ is a strictly ascending chain of  closed subgroups in an abelian minimax group denoted by $A$. Since $l(H_j)$ is an increasing bounded sequence of integers we may assume that $l(H_j)$ is constant. This implies that $|H_{j+1}:H_j|$ is finite for all $j$. Now let $A_0= \cup_j H_j$. The quotient $A_0/H_1$ is an infinite torsion abelian minimax group, which is residualy finite since $H_1$ is closed in $A$. Hence $A_0/H_1$ cannot contain quasicyclic subgroups and hence must be finite. This means that the chain $H_j$ stabilizes, a contradiction. 
\end{proof}

The other direction of Theorem 5 is more subtle. We shall use the following theorem of 
Mal$'$cev, \cite[Theorem 3.6]{inflinear}:

\begin{theorem}\label{malcev}
There is a constant $\mu(n)$ depending only on $n$ such that
any solvable linear group of degree $n$ contains a triangularizable normal subgroup of finite index at most $\mu(n)$.
\end{theorem}

\begin{proof}[Proof of Theorem \ref{sol}]
 (1)$\Rightarrow$(2) This is a special case of the above Lemma.
 \medskip
 
(2)$\Rightarrow$(1) We first deal with the case when $G$ is nilpotent-by-abelian. Let $N=[G,G]$. By assumption $N$ is nilpotent, hence $N$ is minimax if and only if $N/[N,N]$ is minimax. In particular, if $G$ is not minimax then the metabelian group $G/[N,N]$ is not minimax and since all finitely generated metabelian groups are residually finite (see Hall's classic paper \cite{Hall}), we may replace $G$ by $G/[N,N]$ and assume $G$ is metabelian. Now \cite{Hall} shows that there is a finite set of primes $\pi$ such that $[G,G]$ has a free abelian group $B$ such that $G/B$ is $\pi$-torsion. Moreover if $G$ is not minimax then $B$ must have infinite rank. Therefore, taking a prime $p\notin\pi$ we have that $[G,G]/[G,G]^p$ is an elementary abelian $p$-group of infinite rank. Since $[G,G]$ is closed in $G$ it follows that $[G,G]^p$ is closed. Now we can pass to a quotient $G/[G,G]^p$ which has an infinite torsion subgroup. Since all finite subgroups are closed, then it is clear that there are strictly ascending chains.
\medskip

We next consider the general solvable case. Suppose that $G$ is finitely generated, residually finite and solvable and that $G$ satisfies the maximal condition on closed subgroups. Let $A$ be the centre of the centralizer of the last non-trivial term of the derived series of $G$. Then $A$ is closed and $G/A$ is residually finite of smaller derived length, therefore we may assume that $G/A$ is minimax by induction. The torsion subgroup of $A$ must be finite otherwise $G$ would contain a strictly ascending chain of finite subgroups contradicting the maximal condition on closed subgroups. By factoring out this finite subgroup  we may assume that $A$ is torsion-free.

Using Lemma \ref{Lemma20} let $D$ be the unique normal subgroup of $G$ minimal amongst the closed normal subgroups $B$ of $G$ such that $B\subset A$ and $A/B$ is torsion-free, and minimax. 
Then $G/D$ is a residually finite minimax group.
We claim that $D=1$ which implies the result.

Suppose by way of contradiction that $D$ is non-trivial so $D$ is infinite since $D \leq A$ and $A$ is torsion free. As in the proof of Lemma \ref{Lemma20} the ascending condition on closed subgroups implies that the closure $\overline D$ of $D$ in the profinite topology of $G$ is $d$-generated for some $d \in \mathbb N$. For each prime $p$ let $D_p$ denote the intersection of all the normal subgroups $E$ of finite index in $G$ such that $D/D\cap E$ is an elementary abelian $p$-group. Since $D$ is closed in $G$, $D_p \leq D$  and $\dim_{\mathbb F_p} D/D_p \leq d$ for all $p$.  In addition when $D \not =1$ then $D>D_p$ for at least one prime $p$ because $G$ is residually finite. 

We consider first the case when there are infinitely many primes $p$ such that $D_p<D$. 
In that case let $K = \cap_p D_p$, then $D/K$ is infinite. Mal\'cev's Theorem \ref{malcev} applied to the action of $G/K$ on each $\mathbb F_p$-vector space $D/D_p$ gives that  that $G/C_G(D/D_p) \leq GL(d, \mathbb F_p)$ is (nilpotent of class at most $d$)-by-abelian-by-(finite of size at most $\mu(d)$). Let $U_p/C_G(D/D_p)$ be the Fitting subgroup of $G/C_G(D/D_p)$. So $G/U_p$ is abelian by (finite of size at most $\mu(d)$) while $U_p$ acts on $D/D_p$ as a nilpotent group of class at most $d$.
Let $W= \cap_p U_p$, then $G/W$ is abelian-by-finite while $W$ acts on $D/K$ as a nilpotent group of class at most $d$. Let $L/D$ be the Fitting subgroup of  $G/D$. Since  $G/D$ is minimax and so nilpotent-by-abelian-by-finite, it follows  that $G/L$ is abelian-by-finite. 

So $G/(W \cap L)$ is abelian-by-finite while $(W \cap L)/D$ is a nilpotent group which acts nilpotenly on $D/K$. Hence $(W \cap L)/K$ is nilpotent and so $G/K$ is nilpotent-by-abelian-by-finite.

Since $G/K$ is also residually finite, $G/K$ must be minimax by the first case we considered. The maximal condition on closed subgroups implies that the torsion subgroup $T/K$ of $D/K$ is finite,  this means that $D/T$ is torsion free infinite and minimax, and contradicts the minimality of $D$.

 The next case to consider is that there are only finitely many primes $p$ for which $D_p<D$. Since $D$ is infinite there must be a prime $p$ such that the Sylow $p$ subgroup $\overline L_p$ of $\overline D$ is infinite. Recall that  $\overline D$ is $d$ generated, therefore $\overline L_p$ is isomorphic to $F\oplus\Z^k_p$ for some finite $p$-group $F$ and some $k \in [1,d]$. 
Let $D/E$ be the image of $D$ in $\overline L_p$, i.e $E$ is the intersection of all normal subgroups $Y$ of finite index in $G$ such that $D/ D \cap Y$ is a $p$-group.  

Mal$'$cev's theorem implies that $G$ acts on $\overline L_p$ as a nilpotent-by-abelian-by-finite group, this time by appeal to the case of characteristic zero. A similar argument as the case above gives that $G/E$ is a residually finite, nilpotent by abelian by finite group and $D/E$ is infinite. 
This again provides a contradiction to the minimality of $D$.
Therefore $D=1$ after all, and $G$ is minimax.

\medskip

For the general case of an elementary amenable group $G$ we proceed as follows.
Define the hierarchy of classes of elementary amenable groups in the following way. For each ordinal $\alpha$, let $\mathfrak X_\alpha$ be the class of groups defined by
\begin{enumerate}
\item for $\alpha=0$ then $\mathfrak X_0$ is the class of all finitely generated abelian-by-finite groups;
\item
if $\alpha$ is a successor ordinal then $\mathfrak X_\alpha$ consists of the class of all groups which are extensions of a locally $\mathfrak X_{\alpha -1}$-group by an $\mathfrak X_0$-group;
\item
if $\alpha$ is a (non-zero) limit ordinal then $\mathfrak X_\alpha$ consists of all groups belonging to some $\mathfrak X_\beta$ for some $\beta<\alpha$.
\end{enumerate}
Let $G$ be a finitely generated residually finite elementary amenable group with the maximal condition on closed subgroups. If $G$ is not virtually solvable then there exists a least ordinal $\alpha$ such that $G$ has a finitely generated subgroup $H$ in $\mathfrak X_\alpha$ which is not virtually solvable. The minimality of $\alpha$ implies that $\alpha$ is a successor ordinal. Then $H$ has a normal subgroup $K \in \mathfrak X_{\alpha -1}$, while the quotient $H/K$ is abelian-by-finite. The choice of $\alpha$ ensures that $K$ is locally (virtually solvable). Note that $K$ cannot be virtually solvable itself because this would imply that the extension $H$ is also virtually solvable. 

We consider two cases:

Case 1: $K$ is in the variety generated by some of its finitely generated subgroups, say $L$. Since $L$ is virtually solvable it follows that $K$ has a normal solvable subgroup $U$ such that $K/U$ is in the variety generated by some finite group. In particular $K/U$ is locally finite. Let $\bar U$ be the closure of $U$ in the topology of $K$ induced from the profinite topology of $G$. Then $\bar U$ is solvable and $K/ \bar U$ is locally finite. If $K/ \bar U$ is finite then $K$ is virtually solvable, contradition. If $K/ \bar U$ is infinite then $K$ fails the maximal condition on closed subgroups.  

Case 2: $K$ is not in the variety generated by any of its finitely generated subgroups. 
Inside $K$ it is therefore possible to choose a strictly ascending chain $1<K_0<K_1<K_2<\dots$ of  finitely generated virtually solvable subgroups each of which does not belong to the variety of groups generated by its predecessor. Since the closure (in the ambient group $G$) of any $K_i$ belongs to the variety that $K_i$ generates it follows that the closures of the subgroups in this chain also ascend strictly. This contradicts the ascending chain condition and so we deduce that $G$ is virtually solvable after all. The result now follows from the solvable case already considered.
\end{proof}

\bigskip

\emph{Acknowledgements.} The authors wish to thank  A. Garrido, M. Kassabov, R. Sauer, A. Thom and J. Wilson for helpful comments.

\bigskip

\noindent 
\small{A. Kar, University of Southampton, a.kar@soton.ac.uk \\ 
P. H. Kropholler,  University of Southampton, p.h.kropholler@soton.ac.uk \\ 
N. Nikolov, University of Oxford, nikolov@maths.ox.ac.uk}

\bibliography{KKN}
\bibliographystyle{abbrv}

\end{document}